\documentclass[a4paper, 11pt]{article}

\usepackage[T1]{fontenc}
\usepackage[utf8]{inputenc}

\usepackage{amsfonts,amsthm,enumitem,amsmath}
\usepackage{graphicx,caption,amssymb,framed,thm-restate}
\usepackage{tikz}
\usetikzlibrary{calc,decorations}
\usepackage{soul,changepage,comment,mathtools}

\definecolor{bordeaux}{RGB}{100,0,50}
\definecolor{darkblue}{RGB}{25, 25, 112}
\usepackage[hyphens]{url} 
\usepackage[linktoc = all, hidelinks, colorlinks, unicode=true]{hyperref} 
\usepackage[capitalise, compress, nameinlink, noabbrev]{cleveref} 
\hypersetup{linkcolor={{blue!70!black}}, citecolor={black}, urlcolor={blue!70!black}}
\hypersetup{linkcolor=bordeaux, citecolor = darkblue, urlcolor = darkblue}

\newtheorem{lemma}{Lemma}[section]
\newtheorem{theorem}[lemma]{Theorem}

\newtheorem{claim}[lemma]{Claim}

\usepackage{titlesec}
\titlespacing*{\section}{0pt}{0.8ex plus 0.2ex minus 0.2ex}{0.5ex}
\titlespacing*{\subsection}{0pt}{0.2ex}{0ex}
\setlist[enumerate]{topsep=2pt}

\newcommand\claimproofend{\renewcommand{\qedsymbol}{$\boxdot$}
\end{proof}
\renewcommand{\qedsymbol}{$\square$}}

\title{Packing subdivisions into regular graphs}
\author{Richard Montgomery\thanks{Mathematics Institute, University of Warwick, Coventry, UK. Email: {\tt
richard.montgomery@warwick.ac.uk}. Supported by the European Research Council (ERC) under the European Union Horizon 2020 research and innovation programme (grant agreement No.\ 947978).} \and Kalina Petrova\thanks{
Institute of Science and Technology Austria (ISTA), Klosterneuburg, Austria. Email: {\tt
kalina.petrova@ist.ac.at}. This project has received funding from the European Union’s Horizon 2020 research and innovation programme
under the Marie Sk{\l}odowska-Curie grant agreement No 101034413.\includegraphics[width=5.5mm, height=4mm]{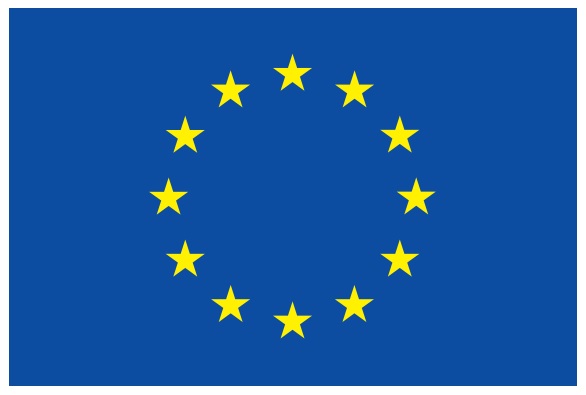}}
\and Arjun Ranganathan\thanks{Department of Mathematics, University College London, London, UK. Email:
 {\tt arjun.ranganathan.24@ucl.ac.uk}} \and Jane Tan\thanks{All Souls College, University of Oxford, Oxford, UK. Email: {\tt
jane.tan@maths.ox.ac.uk}}}
\date{}

\usepackage[left = 2.5cm, right = 2.5cm, top = 2.5cm, bottom = 2.5cm, headsep = 12pt, headheight = 15pt]{geometry}
\makeatletter
\newcommand{\labelinthm}[1]{%
   \label{temp#1}
   \protected@write \@auxout {}{\string \newlabel{#1}{{\emph{\ref{temp#1}}}{\thepage}{\emph{\ref{temp#1}}}{temp#1}{}} }
}
\makeatother
\newcounter{propcounter}

\newcommand{\eps}{\varepsilon}

\newcommand{\bP}{\mathbb{P}}
\newcommand{\bE}{\mathbb{E}}

\begin{document}

\maketitle

\begin{abstract}
We show that, for any graph $F$ and $\eta>0$, there exists a $d_0=d_0(F,\eta)$ such that every $n$-vertex $d$-regular graph with $d \geq d_0$ has a collection of vertex-disjoint $F$-subdivisions covering at least $(1-\eta)n$ vertices. This verifies a conjecture of Verstra\"ete from 2002 and improves a recent result of Letzter, Methuku and Sudakov which additionally required $d$ to be at least polylogarithmic in $n$.
\end{abstract}


\section{Introduction}\label{sec:intro}

The operation of \emph{subdividing an edge} $uv$ of a graph is performed by replacing $uv$ with a new degree 2 vertex incident to $u$ and $v$. If a graph $G$ can be obtained from a graph $F$ by a sequence of edge subdivisions -- equivalently, by replacing edges of $F$ with internally vertex-disjoint paths -- then we say that \emph{$G$ is a subdivision of $F$} (an \emph{$F$-subdivision}, for short) or that \emph{$G$ is a topological $F$} (a \emph{$TF$}). The latter terminology reflects the fact that taking subdivisions preserves many topological properties of the graph, especially pertaining to embeddings. Indeed, the importance of subdivisions in graph theory was cemented early on by the famous theorem of Kuratowski~\cite{kuratowski1930probleme} from 1930 characterising planar graphs as the graphs containing no topological $K_5$ or $K_{3,3}$.

Extremal functions for the average degree forcing different subdivisions have been studied for more than 50 years, with early focus given to topological complete graphs as motivated, for example, by Hadwiger's conjecture. 
Let $d(t)$ be defined as the smallest $d$ such that every graph with average degree more than $d$ has a topological $K_t$. It is not obvious that $d(t)$ should always be finite, but this was established in a classical result of Mader \cite{mader1967homomorphieeigenschaften} in 1967. He furthermore conjectured (as, independently, did Erd\H{o}s and Hajnal~\cite{EH69}) that $d(t)$ is quadratic, to match the easy lower bound given by complete bipartite graphs. After Mader~\cite{mader1972} improved his original bound on $d(t)$ to one exponential in $t$, a rapid flurry of progress in 1994 resulted in the conjectured quadratic upper bound being established by Bollob\'as and Thomason~\cite{bollobas1998proof} using linkages and, in the same few weeks, by Koml\'os and Szemer\'edi~\cite{K-Sz-2} improving their work introducing sublinear expansion~\cite{K-Sz-1}. Sublinear expansion has since become an important tool in extremal graph theory for sparse graphs, on which more details can be found in the recent survey by Letzter~\cite{shohamsurvey}.

While only extremal functions for complete subdivisions were explicitly considered, more generally it follows very directly from Bollob\'as and Thomason's proof, and only slightly less directly from Koml\'os and Szemer\'edi's proof, that, for any $F$, any graph with no $F$-subdivision has average degree $O(|F|+e(F))$. In general, determining the correct asymptotics for the corresponding extremal function is a very interesting open problem. For certain sparse graphs $F$, some progress was made by Haslegrave, Kim and Liu \cite{haslegrave2022extremal}. For complete graphs $F$, the best known is that
\[
(1+o(1))\frac{9t^2}{64}\leq d(t) \leq (1+o(1))\frac{10t^2}{23},
\]
where the upper bound is due to K\"uhn and Osthus~\cite{kuhn2006extremal} and  the lower bound was observed by {\L}uczak (see, for example, \cite{K-Sz-2}) and comes from considering random bipartite graphs.

In 2002, Verstra\"ete \cite{verstraete02} observed that, when $F$ is fixed, all $d$-regular $n$-vertex graphs have a $TF$-packing covering at least $(1/2-o_d(1))n$ vertices. Here, a \emph{$TF$-packing in $G$} is a collection of pairwise vertex-disjoint subgraphs of $G$ which are each a topological $F$. Motivated by this observation, Verstra\"ete conjectured that, for each graph $F$ and $\eta>0$, there is some $d_0$ such that any $n$-vertex $d$-regular graph $G$ with $d\geq d_0$ contains a $TF$-packing covering at least $(1-\eta)n$ vertices. Here, the regularity condition is natural as, for example, the complete bipartite graph with $d$ vertices on one side and $n-d$ vertices on the other can contain vertex-disjoint subdivisions covering at most $O(d|F|)$ vertices, unless $e(F)=0$.

When $F$ is an acyclic graph, Verstra\"ete's conjecture already followed from a slightly earlier result of Kelmans, Mubayi and Sudakov \cite{kelmans2001asymptotically}. Furthermore, when including this conjecture in a problem list in 2003 (see \cite[Conjecture 3.4]{alon2003problems}), Alon noted that it is true when $F$ is a fixed cycle or unicyclic.
As K\"uhn and Osthus~\cite{kuhn2005packings} observed in 2005, when $G$ is dense (i.e., when $d=\Omega(n)$) the conjecture follows from their work finding vertex-disjoint (non-subdivided) copies of any bipartite graph which cover all but extremely few of the vertices in dense regular graphs.

Substantial progress was made very recently by Letzter, Methuku
and Sudakov \cite{LMS25}, who proved that Verstra\"ete's conjecture holds with the additional condition that $d\geq (\log n)^{130}$. This was an application of their work on nearly-Hamilton cycles in sublinear expanders. For more details on this, see \cite{LMS25}, which also has an excellent summary of related work, including the rich literature on packing (non-subdivided) graphs into regular graphs.

Our main theorem verifies Verstra\"ete's conjecture, as follows.

\begin{theorem}\label{thm:packing}
For any $F$ and $\eta>0$, there is some $d_0$ such that the following holds for each $n\geq d\geq d_0$. Every $d$-regular $n$-vertex graph contains a vertex-disjoint collection of subdivisions of $F$ covering at least $(1-\eta)n$ vertices.
\end{theorem}

It is necessary to allow $\eta > 0$ in such a statement. The following example provides an explicit lower bound on $\eta$ in terms of $d$ for most graphs $F$. Let $H$ be obtained from $K_{d+3}$ as follows. First, remove a Hamilton cycle. Then, for distinct vertices $x$, $y$ and $z$ such that $xy,yz$ are edges but $xz$ is not, remove $xy,yz$ and add $xz$. This graph has $d+3$ vertices, $y$ has degree $d-2$, and all other vertices have degree $d$. Let $H_1, \dots, H_d$ be vertex-disjoint copies of $H$ with $y_1, \dots, y_d$ corresponding to $y$. Let $G$ be the union of $H_1, \dots, H_d$, two other vertices $u$ and $v$, and the edges $uy_i$ and $vy_i$ for all $i\in[d]$. Now suppose $F$ is $2$-vertex-connected and that $K_{2,d}$ contains no $F$-subdivision -- for instance, any $2$-vertex-connected $F$ that contains $C_\ell$ for $\ell \geq 5$. Then, as neither $u$ nor $v$ could be in an $F$-subdivision in $G$, any $TF$-packing in $G$ can cover at most $(1-2/|G|)|G|$ vertices. Since $G$ is $d$-regular and $|G|=d(d+3)+2$, in \cref{thm:packing} we require $\eta\geq 2/(d(d+3)+2)$. Where this applies, this slightly improves a similar example of K\"uhn and Osthus~\cite{kuhn2005packings} which shows that $\eta > 1/2d^2$ is necessary for odd $d$ if all the vertices in $F$ are in some cycle.

The proof of \cref{thm:packing} is outlined in \cref{sec:outline}. In \cref{sec:discussion}, we discuss several variants in the context of some further related literature.
The details of the proof are then given across Sections~\ref{sec:prelim} and~\ref{sec:mainproof}: in \cref{sec:prelim} we recount and adapt some known tools to obtain a couple of preliminary results and in \cref{sec:mainproof} we give the main argument.


\subsection{Proof outline}
\label{sec:outline}

We now sketch the proof of \cref{thm:packing}, as illustrated in \cref{fig:outline}. Let $G$ be an $n$-vertex $d$-regular graph. We will split $V(G)$ randomly into $V\cup W$ (where $W$ is a small set), partition $G[V]$ mostly into a collection $\mathcal{P}$ of vertex-disjoint paths, and then find maximally many vertex-disjoint $F$-subdivisions with very many more vertices in $V$ than $W$. Each such $TF$ intersects a path in $\mathcal{P}$ only if it contains the path entirely. If there are many paths in $\mathcal{P}$ not covered by our collection of $F$-subdivisions, then there will be enough edges between the endvertices of the unused paths and the uncovered vertices in $W$ (which is most of $W$ due to the extreme imbalance of the $F$-subdivisions between $V$ and $W$) to allow us to find another $F$-subdivision (by appealing to a certain auxiliary graph whose vertex set is a random subset of the uncovered vertices in $W$).

To give more details: for a small constant $p$, we split $V(G)=V\cup W$ by placing each vertex into $W$ independently at random with probability $p$, and otherwise placing it in $V$. Using the local lemma (see Section~\ref{sec:localpartition}), we can assume that each vertex has roughly the expected degree into $V$ and $W$, $(1-p)d$ and $pd$ respectively, and that $V$ and $W$ are roughly the expected size, $(1-p)n$ and $pn$ respectively.

We then partition all but at most $pn$ vertices of $G[V]$ into a vertex-disjoint collection $\mathcal{P}$ of paths of length $m-1$, where $1/m\ll p$, so that every vertex has at most $4d/m$ neighbours among their endvertices. To do this we use an effective strategy created in part to approach the linear arboricity conjecture (see, e.g., Section~1.1 in~\cite{ferber2020towards} which builds on~\cite{alon2001linear, kelmans2001asymptotically}), as also used recently in~\cite{montgomery2024approximate,LMS25}. This is carried out in Section~\ref{sec:partpartition}, but, in brief, we further partition $V$ into $m$ roughly equal sized pieces $V_1,\ldots,V_m$ using the local lemma so that each $G[V_i,V_{i+1}]$ is an almost-regular bipartite graph and thus contains a large matching, $M_i$ say. Putting the matchings $M_1,\ldots,M_{m-1}$ together, they will each be large enough that this union will mostly consist of our desired paths of length $m-1$. Furthermore, such paths will have endvertices only in $V_1\cup V_m$, so from the properties of our partition we will have that all the vertices have few neighbours among them.

The focus is then on incorporating the majority of paths in $\mathcal{P}$ into $F$-subdivisions while never using many vertices in $W$; since this will cover most of $V$, which in turn covers most of $G$ already, we do not also need to cover most of the vertices in the small set $W$. To see how we find such $F$-subdivisions, let us first describe how to find just one $F$-subdivision in $G$ of our desired kind.

\begin{figure}
\centering
\includegraphics[scale=0.5]{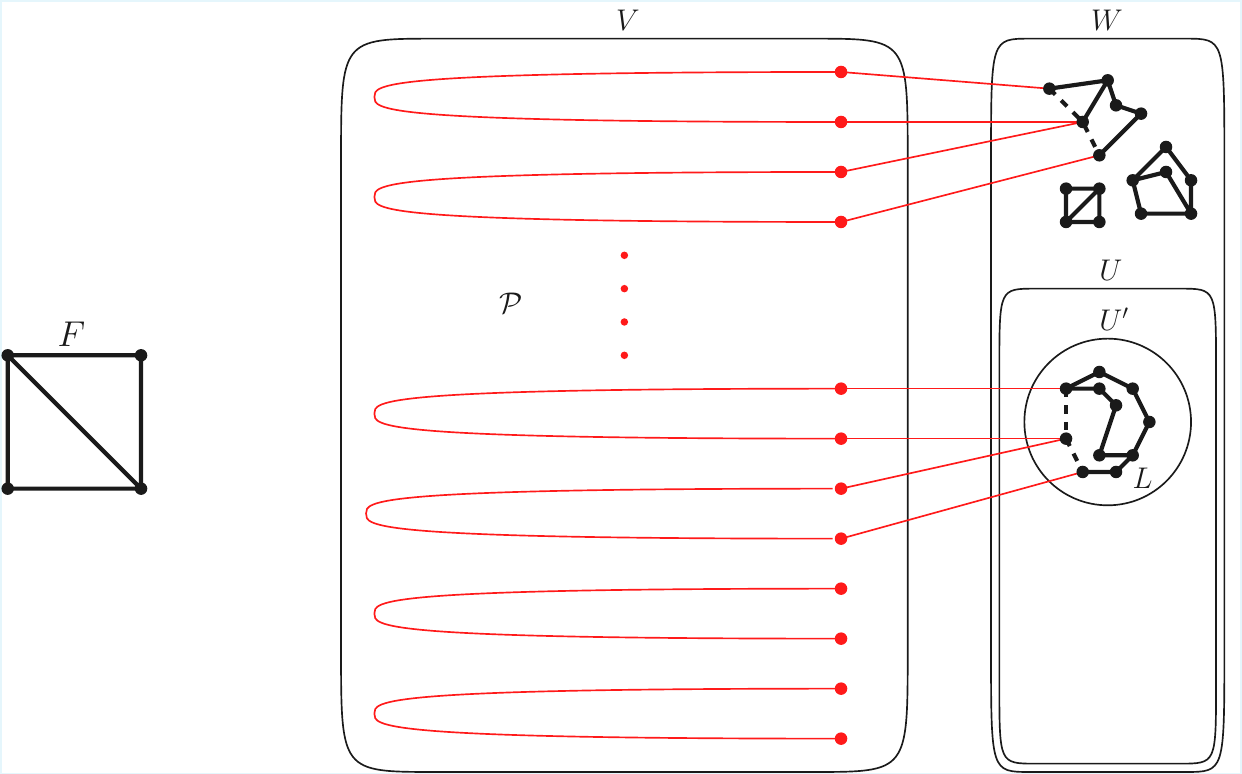}
\caption{Our proof strategy with $F=K_4-e$ and $V(G)$ partitioned into $V$ and $W$, where most of $G[V]$ is partitioned into the collection $\mathcal{P}$ of paths. Subdivisions of $F$ in previous iterations of the auxiliary graph $L$ are drawn in bold, and occupy $W\setminus U$. Each bold edge corresponds to a distinct path in $V$ shown in red (we only show some of these correspondences, indicated by a dashed line). The current auxiliary graph $L$ is drawn, with vertex set $U'\subset U$.} 
\label{fig:outline}
\end{figure}

Take an auxiliary graph $L$ with vertex set $W$ and as many edges as possible subject to the condition that $E(L)$ can be matched into $\mathcal{P}$ such that if $xy\in E(L)$ is matched to $P_{xy}$ then $xP_{xy}y$ is a path in $G$ (for some direction through $P_{xy}$). Then, if we can find an $F$-subdivision $F'$ in $L$, we could replace each edge $xy\in E(L)$ with $xP_{xy}y$ to get an $F$-subdivision which uses many more vertices in $V$ than in $W$, and either uses all of a path from $\mathcal{P}$ or none of it. However, $|\mathcal{P}|$ will be much smaller than $|W|$, so $L$ may have average degree below 1. Instead, then, we take a much smaller random subset $U'\subset W$ such that the endvertices of paths in $\mathcal{P}$ still have some neighbours in $U'$. Forming $L$ on $U'$ rather than $W$ will ensure that $L$ has enough edges relative to $|L|$ to contain an $F$-subdivision by the following result of Mader (mentioned along with subsequent improvements in the introduction).

\begin{theorem}[Mader~\cite{mader1967homomorphieeigenschaften}]
\label{thm:mader} For each graph $F$ there is some $d_F$ such that any graph with average degree at least $d_F$ contains an $F$-subdivision.
\end{theorem}

Suppose, then, that we have a maximal vertex-disjoint collection $\mathcal{F}$ of $F$-subdivisions created in such a manner. Let $\mathcal{P}'$ be the collection of paths in $\mathcal{P}$ not used in any of these subdivisions. It will be easy to argue that most of the endvertices of paths in $\mathcal{P}'$ have plenty of neighbours in the set, $U$ say, of uncovered vertices in $W$. If there are many paths in $\mathcal{P}'$, then we form the auxiliary graph $L$ as above but with a small random subset $U'\subset U$ and $\mathcal{P}'$ in place of $\mathcal{P}$. Theorem~\ref{thm:mader} then implies  $L$ contains an $F$-subdivision, contradicting the maximality of $\mathcal{F}$. Thus, our $F$-subdivisions will cover most of the paths in $\mathcal{P}$ and hence most of the vertices in $G$.

The above sketch works well for graphs $F$ with $e(F)\geq 1$ and no isolated vertices; Theorem~\ref{thm:packing} is stated as \cref{lem:no_isolated_vertices_packing} in this case. For general $F$ with $e(F)\geq 1$, we use \cref{lem:no_isolated_vertices_packing} to pack most of $G$ with subdivisions of $F$ with any isolated vertices removed, having originally reserved some vertices. Then, perhaps removing some of the smaller subdivisions in the initial packing, we will have enough vertices to add as isolated vertices to make the subdivisions into $F$-subdivisions.


\subsection{Further discussion}
\label{sec:discussion}

We made no attempt to optimise the quantitative aspects of \cref{thm:packing}, but one can see from our proof that $d_0$ can be taken as some polynomial of $\eta^{-1}$ and $|F| + e(F)$. This type of dependence is inevitable, as the discussion in the paragraph after \cref{thm:packing} shows that $\eta = \Omega(1/d^2)$ is needed, and if $F = K_t$, then a complete bipartite graph $K_{s,s}$ with $s = \binom{\lfloor t/2\rfloor }{2} - 1$ contains no subdivision of $F$. In line with the average degree needed to guarantee a subdivision of $F$, one may suspect that the correct dependency of $d_0$ on $F$ might be $d_0 = O(|F| + e(F))$. This is not sufficient for our proof and we leave it as an open question.

An analogue of \cref{thm:packing} also holds for edge-disjoint subdivisions, that is, that every regular graph $G$ with high enough degrees (relative to a fixed $F$) has a collection of edge-disjoint subdivisions of $F$ that covers almost all of the edges of $G$. This, indeed, follows trivially from Theorem~\ref{thm:mader} which implies the removal of a maximal such collection leaves a graph with low average degree. Our methods could likely find such an edge-disjoint collection of $F$-subdivisions which moreover consists of the union of vertex-disjoint collections covering almost all of the vertices of $G$ (see \cite[Lemma~3.1]{montgomery2024approximate} for how to modify Lemma~\ref{lem:mostlydecompintopaths}), but we do not complicate our proofs to achieve this.

When studying subdivisions of graphs, it is natural to consider restrictions on the lengths of paths subdividing the original edges (the \emph{subdivision-paths}). For example, a subdivision $G$ of $F$ is called \emph{balanced} if every edge in $F$ is subdivided the same number of times to obtain $G$, i.e., every subdivision-path has the same length. Thomassen~\cite{thomassen1984subdivisions} conjectured that, for every $t$, there is some $d_0$ such that every graph with average degree at least $d_0$ has a balanced subdivision of $K_t$. This was confirmed in 2023 by Liu and Montgomery \cite{LM23}, implying an extremal function for finding balanced subdivisions of a fixed $F$. Indeed, it is possible even to find two isomorphic balanced subdivisions \cite{FHLPW23, LTWY23} using only a constant average degree condition.

We cannot expect a $TF$-packing of isomorphic balanced subdivisions if $F$ has a cycle, since for instance this would not be possible in a disjoint union of regular graphs with small size or large girth. However, dropping the isomorphic requirement makes a packing possible. In our proof, the lengths of subdivision-paths in the final packing are controlled by (in fact, equal to a multiple of) the lengths of subdivision-paths taken in the auxiliary graph $L$. Thus, replacing Mader's theorem by a balanced version produces a balanced $TF$-packing that covers nearly all the vertices of $G$. By contrast, the subdivisions in the packings found in \cite{LMS25} are highly unbalanced; one edge is replaced by a very long path, whilst all others are replaced by short paths.

Our proof can also naturally obtain $TF$-packings in which every subdivision-path has length divisible by a fixed $q\in \mathbb{N}$. The subdivision-paths constructed in the sketch in Section~\ref{sec:outline} are divisible by $m+1$, where $m$ is a parameter that can be easily chosen so that $q|(m+1)$. This may be of interest in relation to the line of work on conditions that force subdivisions with prescribed congruence relations on subdivision-path lengths (see, for example, \cite{Thomassen83, AK21, DDS24}). On the other hand, it is easy to see that not all congruence conditions can be realised, for instance in graphs which have no odd cycles.  Finally, here, we note that by applying the results of~\cite{montgomery2015logarithmically} in place of Theorem~\ref{thm:mader}, we could ensure the subdivisions found in Theorem~\ref{thm:packing} each have size $O_{|F|}(\log n)$.


\section{Preliminaries}
\label{sec:prelim}

For $a,b,c\in\mathbb{R}$, we say $a=(1\pm b)c$ if $(1-b)c\leq a\leq (1+b)c$. We write $x\ll y$ to mean that for all choices of $y>0$ there exists some $x_0>0$ such that the statement in question holds for all $x\le x_0$. For any $n\in \mathbb{N}$, let $[n] := \{1,2,\dots,n\}$.
A graph $G$ has $|G|$ vertices, $e(G)$ edges, and maximum degree $\Delta(G)$. For a vertex $v\in V(G)$, $N_G(v)$ is the set of neighbours of $v$ and $d_G(v)$ is the degree of $v$, that is, $d_G(v)=|N_G(v)|$. Given a subset of vertices $U\subseteq V(G)$, define $d_G(v,U) := |N_G(v) \cap U|$, and let $G[U]$ denote the subgraph of $G$ induced by $U$, with $G-U := G[V(G)\setminus U]$. For any two disjoint subsets $X,Y\subseteq V(G)$, let $G[X,Y]$ denote the induced bipartite subgraph of $G$ with vertex classes $X$ and $Y$.


\subsection{Partitioning and the local lemma}\label{sec:localpartition}
To partition vertices, we will use the local lemma of Lov\'asz (see~\cite{alon2016probabilistic}) and a standard version of Chernoff's bound (see, for example, \cite[Corollary 2.2 and Theorem 2.10]{janson2011random}), as follows. %
\begin{lemma}[Lov\'asz local lemma]\label{Lemma_local_symmetric} Let $B_1,B_2,\dots, B_n$ be events in an arbitrary probability space. Suppose that, for each $i\in [n]$, the event $B_i$ is mutually independent of all but at most $\Delta$ of the other events and $\mathbb{P}(B_i)\leq p$.
 Then, if $ep(\Delta+1)\leq 1$, the probability that none of the events $B_i$, $i\in [n]$, occurs is strictly positive.
\end{lemma}

\begin{lemma}[Chernoff's bound]\label{lem:chernoff}
Let $X$ be a random variable with mean $\mu$ which is binomially distributed.
Then, for any $0<\gamma<1$, we have that $\mathbb{P}(|X-\mu|\geq \gamma \mu)\leq 2e^{-\mu \gamma^2/3}$.
\end{lemma}

We now give our partitioning lemma and a brief proof following~\cite{montgomery2024approximate}. This could also be done via the asymmetric local lemma (see Section~2.4.2 in \cite{montgomery2024approximate} for more discussion on this).

\begin{lemma}\label{lem:partitioning} Let $1/d\ll \gamma,p\leq 1/2$. Let $m\in \mathbb{N}$ and $p_1,\dots,p_m\geq p$ with $\sum_{i\in [m]}p_i=1$.
Let $G$ be a graph with $\Delta(G)\leq 2d$ containing non-empty vertex sets $A$ and $V$ with $d_G(v,A)\geq (1- \gamma)d$ for each $v\in V$.
Then, there is a partition $A=A_1\cup\dots\cup A_m$ such that
    \stepcounter{propcounter}
\begin{enumerate}[label = {{\emph{\textbf{\Alph{propcounter}\arabic{enumi}}}}}]
    \item \labelinthm{prop:vxdegs} for each $v\in V$ and $i\in [m]$, $d_G(v,A_i)\geq (1- 2\gamma)p_id$, and,
    \item \labelinthm{prop:setsizes} for each $i\in [m]$, $|A_i|=(1\pm \gamma)p_i|A|$.
\end{enumerate}
Moreover, if $d_G(v,A)\leq (1+\gamma)d$ for each $v\in V$, then we additionally have that
\begin{enumerate}[label = {{\emph{\textbf{\Alph{propcounter}\arabic{enumi}}}}}]\addtocounter{enumi}{2}
    \item \labelinthm{prop:vxdegsmore} for each $v\in V$ and $i\in [m]$, $d_G(v,A_i)\leq (1+ 2\gamma)p_id$.
    \end{enumerate}
\end{lemma}
\begin{proof} Using, from the degree condition as $V\neq\emptyset$, that $|A|\geq (1-\gamma)d\geq d/2$, fix an arbitrary partition of $A=V_1\cup\dots\cup V_r$ into sets of size between $d/2$ and $d$, for some $r\in \mathbb{N}$. Also, define a random partition $A_1,\dots,A_m$ of $A$ by, independently at random, choosing the set containing $v$ so that $\mathbb{P}(v\in A_i)=p_i$ for each $i\in [m]$. Note that $mp\leq 1$, so $m\leq d$.

For each $i\in [m]$ and $j\in [r]$, let $B_{i,j}$ be the event that $|A_i\cap V_j|\ne (1\pm \gamma)p_i|V_j|$. As $\bE|A_i\cap V_j|=p_i|V_j|$, $p_i\geq p$ and $|V_j|\geq d/2$, we have by \cref{lem:chernoff} that $\bP[B_{i,j}]\le 2\exp(-\gamma^2pd/6)$.
For each $v\in V$ and $i\in [m]$, let $B_{v,i}$ be defined as follows. If $d_G(v,A)\leq (1+\gamma)d$, then $B_{v,i}$ is the event that $d_G(v,A_i)\neq (1 \pm 2\gamma)p_id$; otherwise $B_{v,i}$ is the event that $d_G(v, A_i) < (1-2\gamma)p_id$. By \cref{lem:chernoff}, we have that $\bP[B_{v,i}]\le 2\exp(-\gamma^2pd/6)$.

Now, for each $i\in [m]$ and $j\in [r]$, $B_{i,j}$ depends only on the placement of the vertices in $V_j$ in the partition $A=A_1\cup\dots \cup A_m$, and therefore is independent of any events $B_{i',j'}$ and $B_{v,i''}$ if $j'\neq j$ and $N_{G}(v)\cap V_j=\emptyset$, respectively.
Thus, $B_{i,j}$ is independent of all but at most $m+(2d)dm\leq 3d^3$ of the events $B_{i',j'}$ and $B_{v,i''}$.
Similarly, for every $v\in V(G)$ and $i\in [m]$, $B_{v,i}$ is independent of any events $B_{i',j}$ and $B_{v',i''}$ if $N_G(v)\cap V_{j}=\emptyset$ and $N_G(v)\cap N_G(v')=\emptyset$, respectively. Thus, $B_{v,i}$ is independent of all but at most $2dm+(2d)^2m\leq 5d^3$ of the events $B_{i',j}$ and $B_{v',i''}$.

Thus, every event in $\{B_{i,j}:i\in [m],j\in[r]\}\cup\{B_{v,i}:v\in V(G), i\in [m]\}$ is  mutually independent of all but $5d^3$ of the other events. As $1/d\ll \gamma,p$, we have $e(5d^3+1)\cdot 2\exp(-\gamma^2pd/6)\leq 1$. Thus, by \cref{Lemma_local_symmetric}, there is a partition $A=A_1\cup\dots\cup A_m$ such that none of the events $B_{i,j}$ and $B_{v,i}$ occur. Then \ref{prop:vxdegs} holds immediately, \ref{prop:vxdegsmore} holds if $d_G(v,A)\leq (1+\gamma)d$ for each $v\in V$, and, as   $|A_i|=\sum_{j=1}^{r} |A_i\cap V_j|=\sum_{j=1}^{r}(1\pm \gamma)p_i|V_j|=(1\pm \gamma)p_i|A|$ for every $i\in [m]$, \ref{prop:setsizes} also holds.
\end{proof}


\subsection{Partitioning into paths}\label{sec:partpartition}
We now show that a nearly regular subgraph of $G$ has a collection of vertex-disjoint paths of a specified length that cover nearly all the vertices of the subgraph, using an approach sketched and contextualised in Section~\ref{sec:outline}.

\begin{lemma}\label{lem:mostlydecompintopaths} Let $1/d \ll \gamma\ll 1/m,\eps \leq 1$ and $q\geq 1/2$. Let $G$ be an $n$-vertex graph with $\Delta(G) \leq d$ containing a vertex set $V \subset V(G)$ with $|V|=(1\pm \gamma)qn$ and $d_G(v,V)=(1\pm \gamma)qd$ for each $v\in V(G)$. Let $t=\lceil(1-\eps)qn/m\rceil$.

Then, $G[V]$ contains vertex-disjoint paths ${P}_1,\dots, {P}_{t}$, each of length $m-1$, such that each vertex in $G$ has at most $4d/m$ neighbours in $G$ among the endvertices of $P_1,\dots,P_t$.
\end{lemma}
\begin{proof}
Applying \cref{lem:partitioning} to $V$ with $p_i=1/m$ for each $i\in [m]$, we obtain a partition $V=V_1\cup\dots\cup V_m$ such that the following hold.
    \stepcounter{propcounter}
\begin{enumerate}[label = {{{\textbf{\Alph{propcounter}\arabic{enumi}}}}}]
    \item \label{prop:degreesforpaths} For each $v\in V(G)$ and $i\in[m]$, we have $d_G(v,V_i)=(1\pm 2\gamma)qd/m$.
    \item \label{prop:sizesforpaths} For each $i\in [m]$, we have $|V_i|=(1\pm \gamma)(1\pm\gamma)qn/m=(1\pm 3\gamma)qn/m$.
\end{enumerate}
For each $i\in[m-1]$, let $H_i$ denote the bipartite graph $G[V_i,V_{i+1}]$. By \ref{prop:degreesforpaths} and \ref{prop:sizesforpaths}, we have $e(H_i)\ge((1-2\gamma)qd/m)\cdot (1-3\gamma)q n/m\ge (1-5\gamma)q^2nd/m^2$. Hence, by Vizing's theorem, we can find a matching $M_i$ in $H_i$ of size
\begin{align*}
    |M_i|\ge \frac{e(H_i)}{\Delta(H_i)+1}&\overset{\ref{prop:degreesforpaths}}{\ge} \frac{(1-5\gamma)q^2nd}{m^2}\cdot\frac{m}{(1+2\gamma)qd+m}\\
    &\ge \frac{(1-5\gamma)qn}{(1+3\gamma)m}\ge (1-8\gamma)\frac{qn}{m}\overset{\ref{prop:sizesforpaths}}{\ge} (1-12\gamma)\max\{|V_i|,|V_{i+1}|\}.
\end{align*}
Consider $M_1\cup \dots \cup M_{m-1}$, which is the disjoint union of paths. Note that any of these paths with length less than $m-1$ must have an endvertex in $V_i\setminus V(M_{i})$ or $V_{i+1}\setminus V(M_{i+1})$ for some $i\in [m-1]$. Therefore, the number of paths in $M_1\cup \dots \cup M_{m-1}$ with length $m-1$ is at least
\[
|V_1|-12\gamma\cdot \sum_{i=1}^{m-1}(|V_i|+|V_{i+1}|)\overset{\ref{prop:sizesforpaths}}{\ge}(1-3\gamma)\cdot\frac{qn}{m}-2(m-1)\cdot 12\gamma (1+3\gamma)\cdot\frac{qn}{m}\geq (1-\eps)\cdot\frac{qn}{m}.
\]
Thus we can take vertex-disjoint paths $P_1,\dots, P_t$, each with length $m-1$, in $M_1\cup \dots \cup M_{m-1}$. By construction, the endvertices of all these paths are in $V_1\cup V_m$. Thus, by \ref{prop:degreesforpaths}, each vertex of $G$ has at most $2(1+ 2\gamma)qd/m\le 4d/m$ neighbours among these endvertices, as required.
\end{proof}


\section{Proof of Theorem~\ref{thm:packing}}
\label{sec:mainproof}

For technical reasons, we need to be slightly careful with graphs $F$ that have isolated vertices. To deal with these, we first state a result for graphs without isolated vertices, before inferring \cref{thm:packing}. The following lemma is the crux of our argument.
\begin{lemma}
\label{lem:no_isolated_vertices_packing}
Let $F$ be any fixed graph with no isolated vertices and $e(F)\geq 1$. For each $\nu > 0$, there is some $\gamma>0$ and $d_0$ such that the following holds for each $n \geq d \geq d_0$. Every $n$-vertex graph $G$ with $d_G(v)=(1\pm \gamma)d$ for each $v\in V(G)$ contains a vertex-disjoint collection of $F$-subdivisions which covers at least $(1-\nu)n$ vertices.
\end{lemma}

We now derive the main theorem assuming Lemma \ref{lem:no_isolated_vertices_packing}, as noted at the end of~\Cref{sec:outline}.
\begin{proof}[Proof of \cref{thm:packing}] Note that we can assume $e(F)\geq 1$, for otherwise the theorem is trivially true. So with $F$ and $\eta>0$ as given, let $H$ be the graph obtained from $F$ by deleting all of its isolated vertices.
Applying Lemma~\ref{lem:no_isolated_vertices_packing} with $H$ as the fixed graph without isolated vertices and $\nu = \eta / 2$, we obtain $\gamma>0$ and $d_0$ satisfying the statement of the lemma, and we may moreover assume that $\gamma \leq \eta/2$ and $d_0 \geq 2|F| / (\eta\gamma)$.

Now, let $G$ be any $n$-vertex $d$-regular graph with $n \geq d \geq d_0$. Let $S$ be a set of $\lfloor\gamma d\rfloor$ vertices in $G$. Note that $\tilde G:= G- S$ has $d_{\tilde{G}}(v)=(1\pm \gamma)d$ for each $v\in V(\tilde{G})$. By Lemma~\ref{lem:no_isolated_vertices_packing}, we can cover all but at most $\nu n$ vertices of $\tilde{G}$ with vertex-disjoint $H$-subdivisions, $T_1, \dots, T_t$ say. Assume that these are indexed so that $|T_1| \geq \dots \geq |T_t|$.

For each $j\in [t]$, let $z_j = \sum_{i=1}^j \big(|T_i|  + |F| - |H|\big)$, and let $\ell$ be the minimum index such that $z_\ell \geq (1-\eta)n$. We note that such an $\ell$ exists as we have $\sum_{i=1}^{t}|T_i| \geq |\tilde{G}|-\nu n\geq n - \nu n - |S| \geq (1-\eta)n$. If $z_\ell \leq n$, then we can obtain our desired family of subdivisions of $F$ by taking $T_1, \dots, T_\ell$ and adding $|F|-|H|$ distinct isolated vertices from $V(G) \setminus \bigcup_{i\in[\ell]} V(T_i)$ to each $T_i$ with $i \leq \ell$. To finish, assume for a contradiction that $z_\ell > n$. 

By the minimality of $\ell$, we have that $z_{\ell-1} < (1-\eta)n$ (where we define $z_0 :=0$). Then,
$$|T_\ell| + |F| - |H|=z_\ell-z_{\ell-1} > \eta n,$$ and so $|T_\ell| > \eta n/2$ by our assumption on $d_0$. Since $|T_1| \geq \dots \geq |T_\ell|$, that implies $\ell \leq 2/\eta$. But then, using $d\geq d_0 \geq 2|F| / (\eta\gamma)$,
$$z_\ell= \sum_{i=1}^\ell|T_i|+ \sum_{i=1}^{\ell} \big( |F| - |H|\big) \leq \sum_{i=1}^t|T_i|+\frac{2\big( |F| - |H|\big)}{ \eta} \leq \sum_{i=1}^t|T_i|+\lfloor \gamma d\rfloor = \sum_{i=1}^t|T_i|+|S|\leq n,$$
which contradicts $z_\ell > n$.
\end{proof}

We now return to the central lemma, whose proof will now complete the proof of Theorem~\ref{thm:packing}.
\begin{proof}[Proof of Lemma~\ref{lem:no_isolated_vertices_packing}] Given the graph $F$ with no isolated vertices and $e(F)\geq 1$ and $\nu>0$, take $\gamma,\varepsilon,p\in(0,1)$ and $d_0,m\in\mathbb{N}$ satisfying $1/d_0\ll \gamma\ll \varepsilon\ll 1/m\ll p\ll \nu,1/|F|$. We will show that the property in the lemma holds for $d_0$ and $\gamma$. So fix any $n$ and $d$ satisfying $n\geq d\geq d_0$ and let $G$ be any $n$-vertex graph satisfying  $d_G(v)=(1\pm \gamma)d$ for each $v\in V(G)$.

Using Lemma~\ref{lem:partitioning}, take a partition $V(G)=V\cup W$ so that the following hold.
\stepcounter{propcounter}
\begin{enumerate}[label = {{{\textbf{\Alph{propcounter}\arabic{enumi}}}}}]
\item \label{prop:degs_into_Wi}\label{prop:degs_into_V} For each $v\in V(G)$, $d_G(v,V)=(1\pm 2\gamma)(1-p)d$ and $d_G(v,W)=(1\pm 2\gamma)pd$.
\item \label{prop:sizes_of_sets} $|V|=(1\pm \gamma)(1-p)n$ and $|W|=(1\pm \gamma)pn$.
\end{enumerate}

Let $t=\lceil (1-2p)n/m\rceil$ and $\beta=4/m$. By \ref{prop:degs_into_V}, \ref{prop:sizes_of_sets} and Lemma~\ref{lem:mostlydecompintopaths}, $G[V]$ contains vertex-disjoint paths ${P}_1,\dots, {P}_{t}$,
each of length $m-1$, such that each vertex in $G$ has at most $\beta d=4d/m$ neighbours in $G$ among the endvertices of $P_1,\dots,P_t$.

Let $\mathcal{F}$ be a collection of vertex-disjoint $F$-subdivisions in $G$ which maximises $|\mathcal{F}|$ subject to the following holding for each $H\in \mathcal{F}$.
    \stepcounter{propcounter}
\begin{enumerate}[label = {{{\textbf{\Alph{propcounter}\arabic{enumi}}}}}]
        \item For each $j\in [t]$, $V(H)\cap V(P_j)=\emptyset$ or $V(P_j)\subset V(H)$.\label{prop:allornothingpathcontainment}
        \item $|V(H)\cap W|\leq 2|H|/m$.\label{prop:HiuseslittleofWi}
    \end{enumerate}

Let $J$ be the set of $j\in [t]$ such that $V(P_j)\cap \big(\bigcup_{H\in \mathcal{F}}V(H)\big)=\emptyset$. To complete the proof, it will suffice to verify the following claim.

\begin{claim} $|J|\leq pn/m$.
\end{claim}

Note that if the claim is true, then $\mathcal{F}$ is a collection of vertex-disjoint $F$-subdivisions in $G$ which, by \ref{prop:allornothingpathcontainment}, cover all the vertices of $P_i$, $i\in [t]\setminus J$. Hence, this would mean that
\[
\bigg|\bigcup_{H\in \mathcal{F}}V(H)\bigg|\geq (t-|J|)m\geq (1-2p)n-\frac{pn}{m}\cdot m=(1-3p)n\geq (1-\nu)n.
\]
That is, $\mathcal{F}$ covers all but at most $\nu n$ vertices in $G$, as required.

\medskip

\noindent\emph{Proof of claim.}
Suppose, for contradiction, that $|J|>pn/m$.
Let $U=W\setminus \big(\bigcup_{H\in \mathcal{F}}V(H)\big)$. Now, by \ref{prop:HiuseslittleofWi}, and as the graphs in $\mathcal{F}$ are vertex-disjoint,
\begin{equation}\label{eq:W-Uuppbound}
|W\setminus U|=\sum_{H\in \mathcal{F}}|V(H)\cap W|\leq \frac{2}{m}\sum_{H\in \mathcal{F}}|H|\leq \frac{2n}{m}.
\end{equation}
Let $X$ be the set of endvertices of the paths $P_i$, $i\in J$, which have at least $p d/2$ neighbours in $W\setminus U$. Then, as any vertex in $W\setminus U$ has at most $\beta d$ neighbours among the endvertices of the paths $P_i$, $i\in J$, we have that
\begin{equation}\label{eq:Xupperbound}
|X|\leq \frac{|W\setminus U|\cdot \beta d}{pd/2}\overset{\eqref{eq:W-Uuppbound}}{\leq} \frac{4n\cdot \beta d}{pdm}=\frac{\beta^2n}{p}.
\end{equation}

Label vertices so that, for each $i\in J$, $P_i$ is an $x_i,y_i$-path.
Let $J'$ be the set of $i\in J$ for which $x_i$ and $y_i$ are both not in $X$.  Then, as $\beta \ll p$,
\begin{equation}\label{eq:Jprimelower}
|J'|\geq |J|-|X|\overset{\eqref{eq:Xupperbound}}{>} \frac{pn}{m}-\frac{n\beta^2}{p}\geq \frac{pn}{2m}=\frac{\beta pn}{8}.
\end{equation}

For each $i\in J'$, by \ref{prop:degs_into_Wi} and the fact that  $x_i,y_i\notin X$, we have $d_G(x_i, U), d_G(y_i, U) \geq (1-2\gamma)pd-pd/2= (1 - 4\gamma)pd/2$. Then, using \ref{prop:sizes_of_sets}, we apply \cref{lem:partitioning} with $V=\{x_i,y_i:i\in J'\}$, $A=U$, $m'=2$, $d' = (1+ 2\gamma) pd/2$, $\gamma' = 6\gamma$, $p_1=\eps/2p$ (so $p_2 = 1 -p_1$), and $G'$ the graph $G[(\cup_{i\in J'}\{x_i,y_i\}) \cup U]$ with any edges in $G[\cup_{i\in J'}\{x_i,y_i\}]$ and $G[U]$ removed (note that $\Delta(G') \leq \max\{(1+2\gamma)pd,4d/m\} = 2\cdot (1+2\gamma)pd/2=2d'$). Thus, we can take a subset $U'\subset U$ with $|U'|\leq \eps n$ and $d_G(x_i,U'),d_G(y_i,U')\geq \eps d/10$ for each $i\in J'$.
Take an auxiliary graph $L$ with vertex set $U'$ for which the number of edges is maximal such that there is an injection $f:E(L)\to J'$ with the property that, for each $e\in E(L)$, there is a labelling of $e=e_xe_y$ so that $e_xx_{f(e)},e_yy_{f(e)}\in E(G)$.

We now show that $L$ has no $F$-subdivision. Indeed, suppose there is an $F$-subdivision $F'\subset L$. Then, let $H$ be the graph formed from $F'$ by replacing each edge $xy\in E(F')$ with an $x,y$-path with vertex set $\{x,y\}\cup V(P_{f(xy)})$ (possible by the choice of $f$). As $f$ is an injection, the internal vertices of these paths are disjoint, and thus $H$ is also an $F$-subdivision. As $V(L)=U'\subset U$ and the image of $f$ is contained in $J$, we have that $H$ is vertex-disjoint from each graph in $\mathcal{F}$. Finally, as $F$ has no isolated vertices, $|V(H)\cap U|=|F'|\leq 2e(F')\leq 2|H|/m$ and so $H$ satisfies \ref{prop:HiuseslittleofWi}. Note that, by construction, \ref{prop:allornothingpathcontainment} holds for $H$, and thus altogether we find that $\mathcal{F}\cup \{H\}$ contradicts the maximality of $\mathcal{F}$. Therefore, $L$ has no $F$-subdivision.

Thus, as $\eps\ll\beta\ll p\ll 1/|F|$ and $|U'|\leq \eps n$, from \cref{thm:mader} we have that $e(L)\leq |U'|/\beta\leq \beta p n/16$. Then,

\vspace{-0.5cm}

\[
|J'\setminus f(E(L))|= |J'|-e(L)\overset{\eqref{eq:Jprimelower}}{\geq} \beta pn/8-\beta pn/16= \beta p n/16>0,
\]
so we may pick $j\in J'\setminus f(E(L))$. As $j\in J'\setminus f(E(L))\subset J'$, both $x_j$ and $y_j$ have at least $\eps d/10$ neighbours in $U'$. Then, we can choose disjoint
sets $A_x,A_y\subset U'$, each with size $\lfloor \eps d/20\rfloor$, so that $A_x\subset N_G(x_j)$ and $A_y\subset N_G(y_j)$. The maximality property of $L$ implies that, for each $e_x\in A_x$ and $e_y\in A_y$, we have $e_xe_y\in E(L)$, for otherwise we could add $e_xe_y$ to $L$ and extend $f$ by setting $f(e_xe_y)=j$. Thus, $L[A_x,A_y]$ is a complete bipartite graph with $\lfloor \eps d/20\rfloor$ vertices in each part, and therefore is easily seen to contain an $F$-subdivision as $\lfloor \eps d/20\rfloor\geq |F|^2$, contradicting our previous deduction that $L$ has no $F$-subdivision. Thus, $|J|\leq \frac{pn}{m}$, as required.\hspace{3.5cm}$\boxdot$
\end{proof}

\paragraph{Acknowledgements.} This research was carried out at the workshop ``Extremal and Probabilistic Combinatorics'' held in July 2025 at the International Centre for Mathematical Sciences (ICMS) in Edinburgh. We would like to thank the ICMS and the other workshop participants for the wonderful research atmosphere. AR would like to additionally thank the Royal Society for financial support to attend the workshop.


\end{document}